\documentclass{amsart}
\usepackage[all]{xy}
\usepackage{color}
\usepackage{stmaryrd}
\usepackage{amsthm}
\usepackage{amssymb}
\usepackage[colorlinks=true]{hyperref}

\numberwithin{equation}{subsection}

\newtheorem{theorem}{Theorem}[section]
\newtheorem*{theorem*}{Theorem}
\newtheorem{lemma}[theorem]{Lemma}
\newtheorem{proposition}[theorem]{Proposition}
\newtheorem{corollary}[theorem]{Corollary}
\newtheorem*{corollary*}{Corollary}

\theoremstyle{remark}
\newtheorem{definition}[theorem]{Definition}
\theoremstyle{remark}
\newtheorem{example}[theorem]{Example}

\theoremstyle{remark}
\newtheorem{remark}[theorem]{Remark}
\theoremstyle{remark}
\newtheorem{notation}[theorem]{Notation}

\DeclareMathOperator{\Id}{Id}

\newcommand{\too}{\longrightarrow}
\newcommand{\dg}{\mathrm{dg}}
\newcommand{\dgHo}{\mathrm{H}}

\newcommand{\cA}{{\mathcal A}}
\newcommand{\cB}{{\mathcal B}}
\newcommand{\cC}{{\mathcal C}}
\newcommand{\cD}{{\mathcal D}}
\newcommand{\cE}{{\mathcal E}}
\newcommand{\cF}{{\mathcal F}}
\newcommand{\cG}{{\mathcal G}}
\newcommand{\cH}{{\mathcal H}}

\newcommand{\cL}{{\mathcal L}}

\newcommand{\cN}{{\mathcal N}}
\newcommand{\cO}{{\mathcal O}}
\newcommand{\cP}{{\mathcal P}}

\newcommand{\cT}{{\mathcal T}}

\newcommand{\bbA}{\mathbb{A}}

\newcommand{\bbP}{\mathbb{P}}

\newcommand{\bbS}{\mathbb{S}}

\newcommand{\bbN}{\mathbb{N}}
\newcommand{\bbZ}{\mathbb{Z}}

\newcommand{\bfL}{\mathbf{L}}
\newcommand{\bfR}{\mathbf{R}}

\newcommand{\op}{\mathrm{op}} 
\newcommand{\ie}{\textsl{i.e.}\ }
\newcommand{\eg}{\textsl{e.g.}}

\newcommand{\Hmo}{\mathrm{Hmo}}

\newcommand{\perf}{\mathrm{perf}} 

\newcommand{\Hom}{\mathrm{Hom}} 
\newcommand{\rep}{\mathrm{rep}} 
\newcommand{\dgcat}{\mathrm{dgcat}}
\newcommand{\dgdgcat}{\mathrm{dgdgcat}}

\newcommand{\Spt}{\mathrm{Spt}}

\newcommand{\internalcomment}[1]{}

\title[$\bbA^1$-homotopy invariants of dg orbit categories]{$\bbA^1$-homotopy invariants of dg orbit categories}
\author{Gon{\c c}alo~Tabuada}

\address{Gon{\c c}alo Tabuada, Department of Mathematics, MIT, Cambridge, MA 02139, USA}
\email{tabuada@math.mit.edu}
\urladdr{http://math.mit.edu/~tabuada}

\subjclass[2000]{13F60, 14A22, 14F05, 14J60, 19D25, 19D35, 19D55}
\date{\today}

\keywords{Dg orbit category, $\bbA^1$-homotopy, algebraic $K$-theory, cluster category, Kleinian singularities, Fourier-Mukai transform, noncommutative algebraic~geometry.}

\thanks{The author was partially supported by a NSF CAREER Award.}

\begin{document}
\begin{abstract}
Let $\cA$ be a dg category, $F:\cA \to \cA$ a dg functor inducing an equivalence of categories in degree-zero cohomology, and $\cA/F$ the associated dg orbit category. For every $\bbA^1$-homotopy invariant $E$ (\eg\ homotopy $K$-theory, $K$-theory with coefficients, {\'e}tale $K$-theory, and periodic cyclic homology), we construct a distinguished triangle expressing $E(\cA/F)$ as the cone of the endomorphism $E(F)-\Id$ of $E(\cA)$. In the particular case where $F$ is the identity dg functor, this triangle splits and gives rise to the fundamental theorem. As a first application, we compute the $\bbA^1$-homotopy invariants of cluster (dg) categories, and consequently of Kleinian singularities, using solely the Coxeter matrix. As a second application, we compute the $\bbA^1$-homotopy invariants of the dg orbit categories associated to Fourier-Mukai~autoequivalences.
\end{abstract}
\maketitle

\vskip-\baselineskip
\vskip-\baselineskip
\section{Introduction and statement of results}\label{sec:introduction}
\subsection*{Dg orbit categories}
A {\em differential graded (=dg) category $\cA$}, over a base commutative ring $k$, is a category enriched over complexes of $k$-modules; see \S\ref{sec:dg}. Every (dg) $k$-algebra $A$ gives naturally rise to a dg category with a single object. Another source of examples is provided by schemes since the category of perfect complexes $\perf(X)$ of every quasi-compact quasi-separated $k$-scheme $X$ admits a canonical dg enhancement $\perf_\dg(X)$. When $k$ is a field and $X$ is quasi-projective, this dg enhancement is moreover unique; see Lunts-Orlov \cite[Thm.~2.12]{LO}. In what follows, we will denote by $\dgcat(k)$ the category of small dg categories and dg functors.

Let $F:\cA \to \cA$ be a dg functor which induces an equivalence of categories $\dgHo^0(F): \dgHo^0(\cA) \stackrel{\simeq}{\to} \dgHo^0(\cA)$. Motivated by the theory of cluster algebras, Keller introduced in \cite[\S5.1]{Doc} the associated  {\em dg orbit category} $\cA/F^\bbZ$; see \S\ref{sec:orbit}. This dg category has the same objects as $\cA$ and complexes of $k$-modules defined as 
$$ (\cA/F^\bbZ)(x,y):= \mathrm{colim}_{p\geq 0} \bigoplus_{n \geq 0} \cA(F^n(x),F^p(y))\,.$$
The canonical dg functor $\pi:\cA \to \cA/F^\bbZ$ comes equipped with a morphism of dg functors $\epsilon:\pi \Rightarrow \pi \circ F$, which becomes invertible in $\dgHo^0(\cA/F^\bbZ)$, and the pair $(\pi,\epsilon)$ is the solution of a universal problem; consult \cite[\S9]{Doc} for details. Intuitively speaking, the dg category $\cA/F^\bbZ$ encodes all the information concerning the orbits of the ``homotopical $\bbZ$-action'' on $\cA$. Note that in the particular case where $\cA$ is a $k$-algebra $A$, the dg functor $F$ reduces to a $k$-algebra automorphism $\sigma: A \simeq A$ and the dg orbit category $\cA/F^\bbZ$ reduces to the crossed product $k$-algebra $A \rtimes_\sigma \bbZ$.
\subsection*{$\bbA^1$-homotopy invariants}
A functor $E:\dgcat(k) \to \cT$, with values in a  triangulated category, is called an {\em $\bbA^1$-homotopy invariant} if it satisfies the conditions:
\begin{itemize}
\item[(i)] It inverts the {\em Morita equivalences}, \ie the dg functors $\cA \to \cB$ which induce an equivalence of derived categories $\cD(\cA) \stackrel{\simeq}{\to}\cD(\cB)$; see \S\ref{sec:dg}.
\item[(ii)] It sends short exact sequences of dg categories (see \cite[\S4.6]{ICM-Keller})  
to triangles
\begin{eqnarray*}
0 \to \cA \to \cB \to \cC \to 0 &\mapsto&
E(\cA) \to E(\cB) \to E(\cC) \stackrel{\partial}{\to} \Sigma(E(\cA))\,.
\end{eqnarray*}
\item[(iii)] It inverts the dg functors $\cA \to \cA[t]$, where $\cA[t]$ stands for the tensor product of $\cA$ with the $k$-algebra of polynomials $k[t]$.
\end{itemize}
\begin{example}[Homotopy $K$-theory]\label{ex:1}
Weibel's homotopy $K$-theory gives rise to an $\bbA^1$-homotopy invariant $KH:\dgcat(k) \to \Spt$ with values in the homotopy category of spectra; see \cite[\S2]{Fundamental}\cite[\S5.3]{A1-homotopy}. When applied to $A$, resp. to $\perf_\dg(X)$, it agrees with the homotopy $K$-theory of  $A$, resp. of $X$. Given a prime power $l^\nu$, we can also consider mod-$l^\nu$ homotopy $K$-theory $KH(-;\bbZ/l): \dgcat(k) \to \Spt$. 
\end{example}
\begin{example}[$K$-theory with coefficents]
When $1/l \in k$, mod-$l^\nu$ algebraic $K$-theory gives rise to an $\bbA^1$-homotopy invariant $K(-;\bbZ/l^\nu):\dgcat(k) \to \Spt$; see \cite[\S1]{Klein}. In the same vein, when $l$ is nilpotent in $k$, we have the $\bbA^1$-homotopy invariant $K(-)\otimes \bbZ[1/l]: \dgcat(k) \to \Spt$. When applied to $A$, resp. to $\perf_\dg(X)$, these invariants agree with the algebraic $K$-theory with coefficients of $A$, resp. of $X$.
\end{example}
\begin{example}[{\'E}tale $K$-theory]
Dwyer-Friedlander's {\'e}tale $K$-theory gives rise to an $\bbA^1$-homotopy invariant $K^{et}(-;\bbZ/l^\nu):\dgcat(k) \to \Spt$, where $l^\nu$ be a prime power with $l$ odd; see \cite[\S5.4]{A1-homotopy}. When $1/l \in k$ and the $k$-scheme $X$ is regular and of finite type over $\bbZ[1/l]$, $K^{et}(\perf_\dg(X); \bbZ/l^\nu)$ agrees with the {\'e}tale $K$-theory of $X$.
\end{example}
\begin{example}[Periodic cyclic homology]\label{ex:4}
Let $k$ be a field of characteristic zero. Connes' periodic cyclic homology gives rise to an $\bbA^1$-homotopy invariant $HP: \dgcat(k) \to \cD_{\bbZ/2}(k)$ with values in the derived category of $\bbZ/2$-graded $k$-vector spaces; see \cite[\S3]{Fundamental}. When applied to $A$, resp. to $\perf_\dg(X)$, it agrees with the periodic cyclic homology of $A$, resp. of $X$. Moreover, when the $k$-scheme $X$ is smooth, the Hochschild-Kostant-Rosenberg theorem furnish us the following identifications with de Rham cohomology:
\begin{equation}\label{eq:Hd}
HP^+(\perf_\dg(X))\simeq \bigoplus_{n\,\mathrm{even}} H^n_{dR}(X) \quad \,\, HP^-(\perf_\dg(X))\simeq \bigoplus_{n\, \mathrm{odd}} H^n_{dR}(X)\,.
\end{equation}
\end{example}
\subsection*{Statement of results}
Let $\cA$ be a dg category and $F:\cA \to \cA$ a dg functor which induces an equivalence of categories $\dgHo^0(F)$. Our main result is the following:
\begin{theorem}\label{thm:main}
For every $\bbA^1$-homotopy invariant $E: \dgcat(k)\to \cT$, we have the following distinguished triangle:
\begin{equation}\label{eq:triangle-searched}
E(\cA) \stackrel{E(F)-\Id}{\too} E(\cA) \stackrel{E(\pi)}{\too} E(\cA/F^\bbZ) \stackrel{\partial}{\too} \Sigma (E(\cA))\,.
\end{equation}
\end{theorem}
Note that since $F$ is a Morita equivalence, the above condition (i) implies that $E(F)$ is an automorphism. Hence, we obtain an induced $\bbZ$-action on $E(\cA)$. Theorem \ref{thm:main} shows us then that the dg orbit category $\cA/F^\bbZ$ can be thought of as a ``model'' for the orbits of this $\bbZ$-action on $E(\cA)$. 

\begin{corollary}\label{cor:main}
Assume that $KH(\cA)$ agrees with the algebraic $K$-theory of $\cA$.
\begin{itemize}
\item[(i)] We have $KH_n(\cA/F^\bbZ)=0$ for $n <0$ and $KH_0(\cA/F^\bbZ)$ identifies with the cokernel of the endomorphism $K_0(F)-\Id$ of $K_0(\cA)$. The same holds {\em mutatis mutandis} for $K$-theory with coefficients and {\'e}tale $K$-theory.
\item[(ii)] If the category $\dgHo^0(\cA/F^\bbZ)$ is triangulated and idempotent complete, then $KH_0(\cA/F^\bbZ)$ identifies also with the Grothendieck group of $\dgHo^0(\cA/F^\bbZ)$.
\end{itemize}
\end{corollary}
\begin{remark}
The assumption of Corollary \ref{cor:main} holds for example when $\cA$ is a regular $k$-algebra $A$ or the dg category of perfect complexes of a regular $k$-scheme $X$; see Weibel~\cite[Ex.~1.4 and Prop.~6.10]{Weibel-KH}.
\end{remark}
\begin{corollary}\label{cor:HP}
We have the following six-term exact sequence:
\begin{equation}\label{eq:six-term}
\xymatrix@C=3em@R=1.5em{
HP^+(\cA) \ar[rr]^-{HP^+(F)-\Id} && HP^+(\cA)  \ar[d]^-{HP^+(\pi)} \\
HP^-(\cA/F^\bbZ) \ar[u]^-\partial && HP^+(\cA/F^\bbZ) \ar[d]^-\partial \\
HP^-(\cA) \ar[u]^-{HP^-(\pi)} && HP^-(\cA) \ar[ll]^-{HP^-(F)-\Id} \,.
}
\end{equation}
\end{corollary} 
\begin{proof}
If follows from the long exact sequence of cohomology groups associated to the distinguished triangle \eqref{eq:triangle-searched} (with $E=HP$).
\end{proof}
\section{Applications}\label{sec:examples}
In this section we apply Theorem \ref{thm:main} to different choices of $\cA$ and $F$. 
\subsection*{Identity dg functor}
When $F$ is the identity dg functor, the dg orbit category $\cA/F^\bbZ$ can be identified with the tensor product $\cA[t,t^{-1}]$ of $\cA$ with the $k$-algebra of Laurent polynomial $k[t,t^{-1}]$. Moreover, the morphism $E(F)-\Id$ becomes trivial. As a consequence, the triangle \eqref{eq:triangle-searched} splits and we obtain the following result (which was previously obtained in \cite{Fundamental} in the particular case where $k$ is a~field): 
\begin{corollary}[Fundamental theorem]\label{cor:fundamental}
For every $\bbA^1$-homotopy invariant $E$, we have an isomorphism $E(\cA[t,t^{-1}]) \simeq E(\cA) \oplus \Sigma(E(\cA))$.
\end{corollary}
By applying Corollary \ref{cor:fundamental} to the above Examples \ref{ex:1}-\ref{ex:4}, we hence obtain the fundamental theorem in homotopy $K$-theory, in $K$-theory with coefficients, in {\'e}tale $K$-theory, and in periodic cyclic homology!
\begin{remark}
In the particular case of $k$-algebras and quasi-compact quasi-separated $k$-schemes, the fundamental theorem in homotopy $K$-theory, resp. in periodic cyclic homology, was originally obtained by Weibel \cite{Weibel-KH}, resp. Kassel~\cite{Kassel}. Weibel's proof makes use of the Bass-Quillen's fundamental theorem while Kassel's proof makes use of the relation between periodic cyclic homology and infinitesimal~cohomology.
\end{remark}
\begin{remark}[Inner automorphisms]
In the case where $\cA$ is a dg $k$-algebra $A$ and $F$ a inner automorphism $\sigma: A \simeq A, a \mapsto uau^{-1}$, we have $E(F)=\Id$. This follows from the isomorphism of $A\text{-}A$-bimodules ${}_{\Id} \mathrm{B} \stackrel{\sim}{\to} {}_\sigma \mathrm{B}, a \mapsto ua$ (see \S\ref{sec:dg}-\ref{sec:NCmotives}), and from the fact that $E({}_{\Id} \mathrm{B})=\Id$ (see Lemma \ref{lem:universal}). As a consequence, we still obtain a ``fundamental'' isomorphism $E(\cA \rtimes_\sigma \bbZ) \simeq E(A) \oplus \Sigma(E(A))$.
\end{remark}
%
\subsection*{Suspension dg functor}
Assume that the dg category $\cA$ is pretriangulated in the sense of Bondal-Kapranov \cite[\S3]{BK}. In this case, we can choose for $F$ any of the suspension dg functors $\Sigma^n:\cA \to \cA, n \in \bbZ$.
\begin{example}[$2$-periodic complexes]
Let $k$ be a field, $A$ a finite dimensional hereditary $k$-algebra, $\cD^b(A)$ the bounded derived category of finitely generated right $A$-modules, and $\cD^b_\dg(A)$ the canonical dg enhancement of $\cD^b(A)$. As proved by Peng-Xiao in \cite[Appendix]{PX}, $\dgHo^0(\cD^b_\dg(A)/(\Sigma^2)^\bbZ)$ can be identified with the (triangulated) homotopy category of $2$-periodic complexes of projective right $A$-modules.
\end{example} 
\begin{proposition}\label{prop:suspension}
For every $\bbA^1$-homotopy invariant $E$, we have $E(\Sigma^n)=(-1)^n \Id$. Consequently, we obtain the following computations
$$ E(\cA/(\Sigma^n)^\bbZ) \simeq \left\{  \begin{array}{ll} E(\cA) \oplus \Sigma(E(\cA)) & \mathrm{when}\,\,n\,\,\mathrm{is}\,\,\mathrm{even}\\
E(\cA)/2 & \mathrm{when}\,\,n\,\,\mathrm{is}\,\,\mathrm{odd}\,, \end{array} \right.
$$
where $E(\cA)/2$ stands for the cone of the $2$-fold multiple of the identity of $E(\cA)$.
\end{proposition} 
\begin{example}
When $n$ is even, the $2$-periodicity of $HP$ implies that
$$ HP^+(\cA/(\Sigma^n)^\bbZ) \simeq HP^-(\cA/(\Sigma^n)^\bbZ)\simeq HP^+(\cA) \oplus HP^-(\cA)\,.$$
On the other hand, since multiplication by $2$ is an automorphism of $HP(\cA)$, we conclude that $HP(\cA/(\Sigma^n)^\bbZ)=0$ whenever $n$ is odd. 
\end{example}
\begin{remark}[Universal coefficient sequence]
Given an object $b \in \cT$ and $m \in \bbZ$, let us write $E^b_m(-)$ for the functor $\Hom_\cT(\Sigma^m(b),-)$. For example, when $E=KH$ and $b$ is the sphere spectrum $\bbS$, the functor $E_m^b(-)$ identifies with $KH_m(-)$. Note that when $n$ is odd the distinguished triangle $E(\cA) \stackrel{\cdot 2}{\to} E(\cA) \to E(\cA)/2 \to \Sigma(E(\cA))$ gives rise to the following short exact sequence of abelian groups:
$$ 0 \too E^b_m(\cA) \otimes_\bbZ \bbZ/2 \too E^b_m(\cA/(\Sigma^n)^\bbZ) \too \{2\text{-}\mathrm{torsion}\,\,\mathrm{in}\,\,E^b_{m-1}(\cA)\} \too 0\,.$$
\end{remark}
\subsection*{Finite dimensional algebras of finite global dimension}
Let $k$ be an algebraically closed field and $A$ a finite dimensional $k$-algebra of finite global dimension. The Grothendieck group of $\cD^b(A)$ is free and  a canonical basis is given by the Grothendieck classes $[S_1], \ldots, [S_m]$ of the simple right $A$-modules. Therefore, given a dg functor $F: \cD^b_\dg(A) \to \cD^b_\dg(A)$ which induces an equivalence of categories $\dgHo^0(F)$, the associated group automorphism of $K_0(\cD^b(A))$ can be written as an invertible matrix $M(F) \in \mathrm{Mat}_{m \times m}(\bbZ)$.
\begin{proposition}\label{prop:finite}
For every $\bbA^1$-homotopy invariant $E$, the following holds:
\begin{itemize}
\item[(i)] We have a canonical isomorphism $E(\cD^b_\dg(A)) \simeq \bigoplus_{i=1}^m E(k)$.
\item[(ii)] Under (i), the automorphism $E(F)$ is given by the matrix $M(F)$.
\end{itemize}
\end{proposition}
\subsection*{Cluster categories}
Let $k$ be an algebraically closed field, $Q$ a finite quiver without oriented cycles, and $kQ$ the associated path $k$-algebra. In this case, we can choose for $F$ any of the following dg functors
\begin{eqnarray*}
\tau^{-1} \Sigma^n : \cD^b_\dg(kQ) \too \cD^b_\dg(kQ) && n \geq 0\,,
\end{eqnarray*}
where $\tau$ stands for the Auslander-Reiten translation.
\begin{example}\label{ex:cluster1}
When $n=0$ and $Q$ is of Dynkin type $A$, $D$ or $E$, the category $\cC_Q^{(0)}:= \dgHo^0(\cD^b_\dg(kQ)/(\tau^{-1})^\bbZ)$ identifies with the triangulated category of finite dimensional projective modules over the preprojective algebra $\Lambda(Q)$; see~\cite[\S7.3]{Doc}.
\end{example}
\begin{example}\label{ex:cluster2}
When $n \geq 1$, the category $\cC_Q^{(n)}:= \dgHo^0(\cD^b_\dg(kQ)/(\tau^{-1} \Sigma^n)^\bbZ)$ is called the {\em $(n)$-cluster category of $Q$}; see Reiten \cite[\S3.9]{Reiten}. Thanks to the work of Keller \cite[\S4 Thm.~1]{Doc}, $\cC_Q$ is a triangulated category. These categories play nowadays a key role in representation theory of finite dimensional algebras.
\end{example}
%
\begin{corollary}\label{prop:cluster}
For every $\bbA^1$-homotopy invariant $E$, we have the triangles
$$ \bigoplus_{i=1}^m E(k) \stackrel{(-1)^n C_Q - \Id}{\too} \bigoplus_{i=1}^m E(k) \too E(\cD^b_\dg(kQ)/(\tau^{-1}\Sigma^n)^\bbZ) \too \bigoplus_{i=1}^m \Sigma(E(k))\,,$$
where $C_Q$ stands for the Coxeter matrix of $Q$. Moreover, the Grothendieck group\footnote{When $n=0$ we are implicitly assuming that the quiver $Q$ is as in the above Example \ref{ex:cluster1}. Otherwise, as explained by Keller in \cite[\S3]{Doc}, the category $\cC^{(0)}_Q$ may not be triangulated.} of $\cC^{(n)}_Q$ identifies with the cokernel of the endomorphism $(-1)^n C_Q - \Id$ of $\bigoplus_{i=1}^m \bbZ$.
\end{corollary}
\begin{remark}
The second claim of Proposition \ref{prop:cluster} was obtained independently by Barot-Dussin-Lenzing \cite[\S3]{BDL} in the particular case where $n=1$.
\end{remark}
\begin{proof}
It is well-known that $M(\tau^{-1})$ identifies with the Coxeter matrix $C_Q$ of $Q$. Therefore, the proof of the first claim follows from the combination of Theorem \ref{thm:main} with Propositions \ref{prop:suspension} and \ref{prop:finite}. In what concerns the second claim, it follows from Corollary \ref{cor:main} since $kQ$ is a regular $k$-algebra and the cluster categories $\cC_Q^{(n)}$ are not only triangulated but also idempotent complete (see \cite[Prop.1.2]{BMRRT}).
\end{proof}
\subsection*{Kleinian singularities}
Let $k$ be an algebraically closed field of characteristic zero. As explained by Amiot-Iyama-Reiten in \cite{Amiot}, the category $\cC^{(0)}_Q$ of Example \ref{ex:cluster1} is triangle-equivalent to the stable category of maximal Cohen-Macaulay modules over the Kleinian hypersurface associated to the quiver $Q$. For example, when $Q$ is the Dynkin quiver $A_s\colon 1 \to 2 \to \cdots \to s$, $\cC^{(0)}_{A_s}$ identifies with the category $\underline{\mathrm{MCM}}(R_s)$ where $R_s:=k[x,y,z]/(x^{s+1} + yz)$. Roughly speaking, the latter category\footnote{Also usually known as the category of singularities; see Buchweitz \cite{Buchweitz} and Orlov \cite{Orlov,Orlov1}.} encodes all the information concerning the isolated singularity of the Kleinian hypersurface.
\begin{proposition}\label{prop:computation}
We have an isomorphism $K_0(\underline{\mathrm{MCM}}(R_s))\simeq \bbZ/(s+1)$.
\end{proposition}
\begin{remark}\label{rk:reduced}
By definition, the category of singularities $\cD_{\mathrm{sing}}(R_s)$ is defined as the Verdier quotient $\cD^b(R_s)/\perf(R_s)$. Consequently, we have the exact sequence
\begin{equation}\label{eq:ses-last}
K_0(R_s) \too G_0(R_s) \too K_0(\underline{\mathrm{MCM}}(R_s))\,.
\end{equation}
Since $R_s$ is a local ring, $K_0(R_s)\simeq \bbZ$. Moreover, the composition of $K_0(R_s) \to G_0(R_s) $ with the rank map $G_0(R_s) \to \bbZ$ is the identity. This implies that \eqref{eq:ses-last} gives rise to an injective homomorphism $G_0(R_s)/\bbZ \to K_0(\underline{\mathrm{MCM}}(R_s))$. Making use of previous work of Brieskorn \cite{Brieskorn} on class groups, Herzog-Sanders proved in \cite[Thm.~2.1]{HS} that $G_0(R_s)/\bbZ\simeq \bbZ/(s+1)$. Thanks to Proposition \ref{prop:computation}, we hence conclude that $G_0(R_s)/\bbZ \simeq K_0(\underline{\mathrm{MCM}}(R_s))$.
\end{remark}
\begin{remark}[$K$-theory with coefficients]
Making use of Corollary \ref{prop:cluster} and of previous work of Suslin \cite{Suslin}, the author established in \cite{Klein} a complete computation of all the (higher) algebraic $K$-theory with coefficients of the Kleinian singularities.
\end{remark}
\subsection*{A cyclic quotient singularity}
Let $k$ be an algebraically closed field of characteristic zero. Consider the $\bbZ/3$-action on the power series ring $k\llbracket x,y,z\rrbracket$ given by multiplication with a primitive third root of unit. As proved by Keller-Reiten in \cite[\S2]{KR}, the stable category of maximal Cohen-Macaulay modules $\underline{\mathrm{MCM}}(R)$ over the fixed point ring $R:=k\llbracket x,y,z\rrbracket^{\bbZ/3}$ is triangle-equivalence to the $(1)$-cluster category of the generalized Kronecker quiver $Q\!:\!\!\!\xymatrix@C=1.7em@R=1em{1\ar@<0.7ex>[r]\ar[r]\ar@<-0.7ex>[r] & 2}$.
\begin{proposition}\label{prop:computation2}
We have an isomorphism $K_0(\underline{\mathrm{MCM}}(R))\simeq \bbZ/3 \times \bbZ/3$.
\end{proposition}
To the best of the author's knowledge, this computation is new in the literature.
\begin{notation}
In order to simplify the exposition of the next two subsections, we will write $-\otimes -, q_\ast$, and $\Delta_\ast$, instead of $-\otimes^\bfL-,\bfR q_\ast$, and $\bfR \Delta_\ast$,~respectively.
\end{notation}
\subsection*{Line bundles}
Let $k$ be a field, $X$ a smooth projective $k$-scheme, $\cL$ a line bundle on $X$, and $n \in \bbZ$. In this case, we can choose for $F$ the following dg functor:
\begin{eqnarray}\label{eq:equivalence-bundle}
\perf_\dg(X) \too \perf_\dg(X) && \cF \mapsto \cF \otimes \cL[n]\,.
 \end{eqnarray}
\begin{example}
When $n:=\mathrm{dim}(X)$ and $\cL$ is the canonical line bundle $\omega_X := \bigwedge^nT_X^\ast$, the associated dg functor $-\otimes \omega_X[n]$ is called the {\em Serre dg functor of $X$}.
\end{example} 
\begin{remark}
When the canonical or anti-canonical line bundle of $X$ is ample, Bondal-Orlov proved in \cite[Thm.~3.1]{BO} that modulo the autoequivalences of $\perf(X)$ induced by the automorphisms of $X$, all the autoequivalences are as in \eqref{eq:equivalence-bundle}.
\end{remark}
The following example shows that in some particular cases the dg orbit category associated to \eqref{eq:equivalence-bundle} can alternatively be constructed using an automorphism.
\begin{example}[Abelian varieties]
Let $k$ be an algebraic closed field, $\mathrm{A}$ an abelian variety, $\widehat{\mathrm{A}}$ its dual, and $\cP$ the Poincar{\'e} bundle on $\widehat{\mathrm{A}} \times \mathrm{A}$. Given $\alpha \in \widehat{\mathrm{A}}$, consider the line bundle $\cP_\alpha:= \cP_{|\alpha \times \mathrm{A}}$ on $\mathrm{A}$ and the translation automorphism $t_\alpha: \widehat{\mathrm{A}} \simeq \widehat{\mathrm{A}}$. Thanks to the work of Mukai \cite[Thm.~2.2]{Mukai}, the Fourier-Mukai dg functor 
\begin{eqnarray}
\Phi_\cP: \perf_\dg(\mathrm{A}) \too \perf_\dg(\widehat{\mathrm{A}}) && \cF \mapsto q_\ast(p^\ast(\cF)\otimes \cP)\,,
\end{eqnarray}
where $\Delta: X \to X\times X$ is the diagonal morphism and $q,p: X \times X \to X$ the first and second projections, is a Morita equivalence, Moreover, $(-\otimes \cP_\alpha) \Phi_\cP \simeq \Phi_\cP t_\alpha^\ast$; see \cite[\S3.1]{Mukai}. As a consequence, we obtain an induced Morita equivalence between the dg orbit categories $\perf_\dg(\mathrm{A})/(-\otimes \cP_\alpha)^\bbZ$ and $\perf_\dg(\widehat{\mathrm{A}})/(t_\alpha^\ast)^\bbZ$.
\end{example}
Since the $k$-scheme $X$ is regular, homotopy $K$-theory $KH(X)$ agrees with algebraic $K$-theory $K(X)$. Consequently, Theorem \ref{thm:main} (with $E=KH$) combined with Proposition \ref{prop:suspension} give rise to the following distinguished triangle:
$$ K(X) \stackrel{(-1)^nK(-\otimes \cL)-\Id}{\too} K(X) \stackrel{K(\pi)}{\too} KH(\perf_\dg(X)/(-\otimes \cL[n])^\bbZ) \too \Sigma(K(X))\,.$$
\begin{example}[Curves]
When $X$ is a smooth projective curve $C$, we have $K_0(C)\simeq\bbZ \times \mathrm{Pic}(C)$ and the homomorphism $K_0(-\otimes \cL)$ identifies with multiplication by $\cL$. Thanks to Corollary \ref{cor:main}(i), we hence obtain the following computation: 
\begin{equation}\label{eq:computation-Picard}
KH_0(\perf_\dg(C)/(-\otimes \cL[n])^\bbZ) \simeq \left\{  \begin{array}{ll} \bbZ \times \mathrm{Pic}(C)/_{\!\cL=\cO_C} &\mathrm{if}\,\,n\,\,\mathrm{is}\,\,\mathrm{even}\\
\bbZ/2\bbZ \times \mathrm{Pic}(C)/_{\!\cL=\cO_C} & \mathrm{if}\,\,n\,\,\mathrm{is}\,\,\mathrm{odd}\,.\end{array} \right.
\end{equation}
In the particular case where $k$ is algebraically closed, $\cL:=\omega^\ast_C$, and $n \geq 1$, the category $\dgHo^0(\perf_\dg(C)/(-\otimes \omega^\ast_C[n])^\bbZ)$ is known to be triangulated and idempotent complete. This follows from the combination of \cite[\S 9.9 Thm.~6]{Doc} and \cite[Prop.~1.2]{BMRRT} with the fact that the abelian category of coherent $\cO_C$-modules $\mathrm{Coh}(C)$ is hereditary and satisfies the Krull-Schmidt condition. Consequently, thanks to Corollary \ref{cor:main}(ii), the left-hand side of \eqref{eq:computation-Picard} identifies also with the Grothendieck group of the triangulated category $\dgHo^0(\perf_\dg(C)/(-\otimes \omega^\ast_C[n])^\bbZ)$.
\end{example}
In what concerns periodic cyclic homology, we have the following result:
\begin{proposition}\label{prop:line}
When $\cA$ is the dg category $\perf_\dg(X)$ and $F$ the dg functor $-\otimes \cL[n]$, the above $6$-term exact sequence \eqref{eq:six-term} reduces to
$$
\xymatrix@C=2em@R=2em{
\bigoplus_{n\,\mathrm{even}} H^n_{dR}(X)  \ar[rr]^-{(-1)^n(-\cdot \mathrm{ch}(\cL))-\Id} && \bigoplus_{n\,\mathrm{even}} H^n_{dR}(X) \ar[d]^-{HP^+(\pi)} \\
HP^-(\perf_\dg(X)/(-\otimes \cL[n])^\bbZ) \ar[u]^-\partial && HP^+(\perf_\dg(X)/(-\otimes \cL[n])^\bbZ) \ar[d]^-\partial \\
\bigoplus_{n\,\mathrm{odd}} H^n_{dR}(X) \ar[u]^-{HP^-(\pi)} && \bigoplus_{n\,\mathrm{odd}} H^n_{dR}(X) \ar[ll]^-{(-1)^n(-\cdot \mathrm{ch}(\cL))-\Id} \,,
}
$$
where $\mathrm{ch}(\cL) \in \bigoplus_{n} H^{2n}_{dR}(X)$ stands for the Chern character of $\cL$.
\end{proposition}
\subsection*{Spherical objects}
Let $k$ be a field and $X$ a smooth projective $k$-scheme of dimension $d$. Recall from \cite[Def.~1.1]{ST} that an object $\cE \in \perf(X)$ is called {\em spherical} if $\cE \otimes \omega_X \simeq \cE$ and $\Hom_{\perf(X)}(\cE,\cE[\ast])$ agrees with the cohomology $H^\ast(S^d;k)$ of the $d$-dimensional sphere. Thanks to the work of Seidel-Thomas \cite[Thm.~1.2]{ST}, we can then choose for $F$ the associated spherical twist:
\begin{eqnarray*}
& \perf_\dg(X) \stackrel{\Phi}{\too} \perf_\dg(X) & \cF \mapsto q_\ast(p^\ast(\cF) \otimes \mathrm{cone}\left(q^\ast(\cE^\vee)\otimes p^\ast(\cE) \to \Delta_\ast(\cO_X)\right))\,.
\end{eqnarray*}
\begin{proposition}\label{prop:reduction}
For every $\bbA^1$-homotopy invariant $E$, we have the equality $E(\Phi) = \Id - E(\Phi')$, where $\Phi'$ stands for the Fourier-Mukai dg functor $\Phi_{q^\ast(\cE^\vee)\otimes p^\ast(\cE)}$.
\end{proposition}
\begin{remark}[$\bbP$-twists]
Recall from \cite[Def.~1.1]{HT} that an object $\cE \in \perf(X)$ is called a {\em $\bbP^n$-object} if $\cE\otimes \omega_X \simeq \cE$ and $\Hom_{\perf(X)}(\cE,\cE[\ast])\simeq H^\ast(\bbP^n;k)$. Thanks to the work of Huybrechts-Thomas \cite[Prop.~2.6]{HT}, we can also choose for $F$ the associated $\bbP^n$-twist $\Phi$, whose definition is similar to the one above but with $q^\ast(\cE^\vee)\otimes p^\ast(\cE)$ replaced by a slightly more involved object $\cH\in\perf(X \times X)$. Proposition \ref{prop:reduction} holds {\em mutatis mutandis} in this case with $\Phi':=\Phi_\cH$.
\end{remark}
Once again, since the $k$-scheme $X$ is regular, homotopy $K$-theory $KH(X)$ agrees with algebraic $K$-theory $K(X)$. Consequently, Theorem \ref{thm:main} (with $E=KH$) combined with  Proposition \ref{prop:reduction} gives rise to the following distinguished triangle:
$$ K(X) \stackrel{-K(\Phi')}{\too} K(X) \stackrel{K(\pi)}{\too} KH(\perf_\dg(X)/\Phi^\bbZ) \too \Sigma(K(X))\,.$$
It is well-known (see \cite[page 40]{ST}) that the homomorphism $K_0(\Phi')$ is given by
\begin{eqnarray}\label{eq:last}
K_0(X) \too K_0(X) && [\cF] \mapsto \sum_i (-1)^i \mathrm{dim}\, \mathrm{Ext}^i (\cF,\cE) \cdot [\cE]\,.
\end{eqnarray}
Therefore, thanks to Corollary \ref{cor:main}(i), $KH_0(\perf_\dg(X)/\Phi^\bbZ)$ identifies with the cokernel of \eqref{eq:last}. In what concerns periodic cyclic homology, recall from \S\ref{sec:remaining} that $H^\ast_{dR}(X)$ comes equipped with a non-degenerate Mukai pairing $\langle-,-\rangle$. Making use of it, we introduce the following projection homomorphism
\begin{eqnarray}\label{eq:projection}
& \bigoplus_n H^n_{dR}(X) \too \bigoplus_n H^n_{dR}(X) & \alpha \mapsto \langle \alpha, \mathrm{ch}(\cE) \sqrt{\mathrm{Td}_X}\rangle \cdot \mathrm{ch}(\cE) \sqrt{\mathrm{Td}_X}\,,
\end{eqnarray}
where $\sqrt{\mathrm{Td}_X} \in \bigoplus_n H_{dR}^{2n}(X)$ stands for the square root of the Todd class of $X$.
\begin{proposition}\label{prop:spherical}
Via the above identification \eqref{eq:Hd}, $HP^\pm(\Phi')$ agrees with \eqref{eq:projection}. Consequently, Proposition \ref{prop:line} holds {\em mutatis mutandis} with $-\otimes \cL[n]$ replaced by $\Phi$ and $(-1)^n (-\cdot \mathrm{ch}(\cL)) -\Id$ replaced by $(-1)$\eqref{eq:projection}.
\end{proposition}
\subsection*{Related work}
Given a (not necessarily unital) $k$-algebra $A$ and a $k$-algebra automorphism $\sigma:A \simeq A$, Corti{\~n}as-Thom proved in \cite[Thm~7.4.1]{CT} that for every $M_\infty$-stable, excisive and homotopy invariant homology theory $E$ of $k$-algebras, we have the following distinguished triangle:
\begin{equation}\label{eq:triangle-CT}
\Sigma^{-1}(E(A)) \stackrel{\partial}{\too} E(A) \stackrel{E(\sigma)-\Id}{\too} E(A) \too E(A \rtimes_\sigma \bbZ)\,.
\end{equation}
Their proof is an adaptation of Cuntz's work \cite{Cuntz}, who extended the original Pimsner-Voiculescu's $6$-term exact sequence \cite{Pimsner} from the realm of operator $K$-theory to the realm of topological algebras. Our proof of Theorem \ref{thm:main} is radically different! (Corti{\~n}as-Thom's arguments don't extend to the dg setting). It is based on a careful study of the derived category of $\cA/F^\bbZ$ using the recent theory of noncommutative motives; see \S\ref{sec:proof}. Due to its generality, which is exemplified in \S\ref{sec:examples}, we believe that Theorem \ref{thm:main} (and Corollaries \ref{cor:main} and \ref{cor:HP}) will be useful for all those mathematicians whose research comes across (dg) orbit categories.
\medbreak\noindent\textbf{Acknowledgments:} The author is very grateful to Guillermo Corti{\~n}as for explanations about the distinguished triangle \eqref{eq:triangle-CT} during an entertaining afternoon walk on Tiananmen Square. The author also would like to thank Christian Haesemeyer for useful discussions concerning Proposition \ref{prop:computation}  and Remark \ref{rk:reduced}.
\section{Preliminaries}
Throughout the article, $k$ will be a base commutative ring. All adjunctions will be displayed vertically with the left  (resp. right) adjoint on the left (resp. right) hand side. Unless stated differently, all tensor products will be taken over $k$.
\subsection{Dg categories}\label{sec:dg}
Let $\cC(k)$ be the category of cochain complexes of $k$-modules. A {\em differential graded (=dg) category $\cA$} is a $\cC(k)$-enriched category and a {\em dg functor} $F:\cA\to \cB$ is a $\cC(k)$-enriched functor. Given dg functors $F,G: \cA \to \cB$, a {\em morphism of dg functors} $\epsilon: F \Rightarrow G$ consists of a family of degree-zero cycles $\epsilon_x \in \cB(F(x),G(x)), x \in \cA$, satisfying the equalities $G(f)\circ \epsilon_x = \epsilon_y\circ F(f)$ for all $f \in \cA(x,y)$; consult Keller's ICM survey \cite{ICM-Keller} for further details. Recall from \S\ref{sec:introduction} that we denote by $\dgcat(k)$ the category of small dg categories and dg functors.

Let $\cA$ be a dg category. The category $\dgHo^0(\cA)$ has the same objects as $\cA$ and $\dgHo^0(\cA)(x,y):=H^0\cA(x,y)$. The opposite dg category $\cA^\op$ has the same objects as $\cA$ and $\cA^\op(x,y):=\cA(y,x)$. A {\em right $\cA$-module} is a dg functor $M:\cA^\op \to \cC_\dg(k)$ with values in the dg category $\cC_\dg(k)$ of cochain complexes of $k$-modules. Let $\cC(\cA)$ be the category of right $\cA$-modules. As explained in \cite[\S3.1]{ICM-Keller}, the dg structure of $\cC_\dg(k)$ makes $\cC(\cA)$ into a dg category $\cC_\dg(\cA)$. Given an object $x \in \cA$, let us write $\widehat{x}: \cA^\op \to \cC_\dg(k)$ for the associated Yoneda right $\cA$-module defined by $y \mapsto \cA(y,x)$. This assignment gives rise to the Yoneda dg functor $\cA \mapsto \cC_\dg(\cA), x \mapsto \widehat{x}$. The {\em derived category $\cD(\cA)$ of $\cA$} is the localization of $\cC(\cA)$ with respect to (objectwise) quasi-isomorphisms. Its subcategory of compact objects will be denoted by $\cD_c(\cA)$. 

Every dg functor $F:\cA \to \cB$ gives rise to the following adjunctions
$$
\xymatrix{
\cC(\cB) \ar@<1ex>[d]^-{F^\ast} && \cD(\cB) \ar@<1ex>[d]^-{F^\ast} \\
\cC(\cA) \ar@<1ex>[u]^-{F_\ast} && \cD(\cA) \ar@<1ex>[u]^-{\bfL F_\ast}\,,
}
$$
where $F^\ast$ is defined by pre-composition with $F$ and $F_\ast$ (resp. $\bfL F_\ast$) is its left adjoint (resp. derived left adjoint). A dg functor $F:\cA\to \cB$ is called a {\em Morita equivalence} if $\bfL F_\ast:\cD(\cA) \stackrel{\simeq}{\to} \cD(\cB)$ is an equivalence of categories. As proved in \cite[Thm.~5.3]{IMRN}, $\dgcat(k)$ admits a Quillen model structure whose weak equivalences are the Morita equivalences. Let us write $\Hmo(k)$ for the associated homotopy category.

The {\em tensor product $\cA\otimes\cB$} of dg categories is defined as follows: the set of objects is the cartesian product of the sets of objects of $\cA$ and $\cB$ and $(\cA\otimes\cB)((x,w),(y,z)):= \cA(x,y) \otimes \cB(w,z)$. As explained in \cite[\S2.3 and \S4.3]{ICM-Keller}, this construction gives rise to symmetric monoidal categories $(\dgcat(k), -\otimes-, k)$ and $(\Hmo(k), -\otimes^\bfL-,k)$.

An {\em $\cA\text{-}\cB$-bimodule} is a dg functor $\mathrm{B}:\cA \otimes^\bfL \cB^\op\to \cC_\dg(k)$ or equivalently a right $(\cA^\op \otimes \cB)$-module. Standard examples are the ``diagonal'' $\cA\text{-}\cA$-bimodule
\begin{eqnarray}\label{eq:bimodule-Id}
{}_{\Id} \mathrm{B}:\cA \otimes^\bfL \cA^\op \too \cC_\dg(k) && (x,y) \mapsto \cA(y,x)
\end{eqnarray}
and more generally the $\cA\text{-}\cB$-bimodule
\begin{eqnarray}\label{eq:bimodules111}
{}_F\mathrm{B}:\cA\otimes^\bfL \cB^\op \too \cC_\dg(k) && (x,w) \mapsto \cB(w,F(x))
\end{eqnarray}
associated to a dg functor $F:\cA \to \cB$. Let us denote by $\rep(\cA,\cB)$ the full triangulated subcategory of $\cD(\cA^\op \otimes^\bfL \cB)$ consisting of those $\cA\text{-}\cB$-bimodules $\mathrm{B}$ such that $\mathrm{B}(x,-) \in \cD_c(\cB)$ for every object $x \in \cA$. 
\subsection{Noncommutative motives}\label{sec:NCmotives}
As proved in \cite[Cor.~5.10]{IMRN}, there is a natural bijection between $\Hom_{\Hmo(k)}(\cA,\cB)$ and the set of isomorphism classes of $\rep(\cA,\cB)$. Under this bijection, the composition law of $\Hmo(k)$ corresponds to
\begin{eqnarray}\label{eq:bimodules11}
\rep(\cA,\cB) \times \rep(\cB,\cC) \too\rep(\cA,\cC) && (\mathrm{B},\mathrm{B}')\mapsto \mathrm{B} \otimes^\bfL_\cB \mathrm{B}'
\end{eqnarray}
and the identity of an object $\cA \in \Hmo(k)$ to the isomorphism class of ${}_{\Id} \mathrm{B}$. Since the $\cA\text{-}\cB$-bimodules \eqref{eq:bimodules111} belong to
$\rep(\cA,\cB)$, we have a symmetric monoidal functor
\begin{eqnarray}\label{eq:functor1}
\dgcat(k) \too \Hmo(k) & \cA \mapsto \cA & F \mapsto {}_F\mathrm{B}\,.
\end{eqnarray}
The category of {\em noncommutative motives} $\Hmo_0(k)$ has the same objects as $\Hmo(k)$ and abelian groups of morphisms given by $\Hom_{\Hmo_0(k)}(\cA,\cB):=K_0\rep(\cA,\cB)$, where $K_0\rep(\cA,\cB)$ stands for the Grothendieck group of the triangulated category $\rep(\cA,\cB)$. The composition law is induced from the bi-triangulated functor \eqref{eq:bimodules11}. In what concerns the symmetric monoidal structure, it is induced by bi-linearity from $\Hmo(k)$. Note that we have a well-defined symmetric monoidal functor
\begin{eqnarray}\label{eq:nat2}
\Hmo(k) \too \Hmo_0(k) &\cA \mapsto \cA& \mathrm{B} \mapsto [\mathrm{B}]\,.
\end{eqnarray}
Let us denote by $U: \dgcat(k) \to \Hmo_0(k)$ the composition of \eqref{eq:functor1} with \eqref{eq:nat2}. For further details on (the category of) noncommutative motives, we invite the reader to consult the survey article \cite{survey} as well as Kontsevich's talks \cite{Miami,IAS}.
\subsection{Dg orbit categories}\label{sec:orbit}
Let $\cA$ be a dg category and $F:\cA \to \cA$ a dg functor. The dg category $\cA/F^\bbN$ has the same objects as $\cA$ and complexes of $k$-modules
\begin{equation}\label{eq:decomp}
(\cA/F^\bbN)(x,y):= \bigoplus_{n \geq 0} \cA(F^n(x),y)\,.
\end{equation}
Given objects $x,y,z$ and morphisms
\begin{eqnarray*}
\mathrm{f}=\{f_n\}_{n \geq 0} \in \bigoplus_{n \geq 0} \cA(F^n(x),y) && \mathrm{g}=\{g_n\}_{n \geq 0} \in \bigoplus_{n \geq 0} \cA(F^n(y),z)\,,
\end{eqnarray*}
the $m^{\mathrm{th}}$-component of the composition $\mathrm{g} \circ \mathrm{f}$ is defined as $\sum_n (g_n \circ F^n(f_{m-n}))$. 

For every object $x \in \cA$, let us denote by $\epsilon'_x=\{\epsilon'_{x,n}\}_{n \geq 0} \in \bigoplus_{n \geq 0} \cA(F^n(x), F(x))$ the morphism in $\cA/F^\bbN$ from $x$ to $F(x)$ such that $\epsilon'_{x,1}=\Id$ and $\epsilon'_{x,n}=0$ for $n \neq 1$.
Note that $\cA/F^\bbN$ comes equipped with the canonical dg functor
\begin{eqnarray*}
\pi': \cA \too \cA/F^\bbN & x \mapsto x & f \mapsto \mathrm{f} = \{f_n\}_{n \geq 0}\,,
\end{eqnarray*}
where $f_0=f$ and $f_n=0$ for $n \neq 0$, and that the assignment $x \mapsto \epsilon'_x$ gives rise to a morphism of dg functors $ \epsilon': \pi' \Rightarrow \pi' \circ F$. Note also that $F$ extends to a well-defined dg functor $F: \cA/F^\bbN \to \cA/F^\bbN$ and that $F(\epsilon'_x)=\epsilon'_{F(x)}$ for every object $x \in \cA$.
\begin{definition}
Let $\cA/F^\bbZ$ be the dg category with the same objects as $\cA$ and with complexes of $k$-modules defined as $(\cA/F^\bbZ)(x,y):= \mathrm{colim}_{p\geq 0} (\cA/F^\bbN)(x,F^p(y))$, where the colimit is induced by the morphisms $\epsilon'_{F^p(y)}: F^p(y) \to F^{p+1}(y)$. The composition law is determined by the following morphisms:
$$
(\cA/F^\bbN)(y,F^{p'}(z)) \otimes (\cA/F^\bbN)(x,F^p(y)) \to (\cA/F^\bbN)(x,F^{p+p'}(z)) \,\,\, (\mathrm{g},\mathrm{f}) \mapsto F^p(\mathrm{g}) \circ \mathrm{f}.
$$
\end{definition}
Note that $\cA/F^\bbZ$ comes equipped with a canonical dg functor $\iota: \cA/F^\bbN \to \cA/F^\bbZ$. Let us denote by $\pi: \cA \to \cA/F^\bbZ$ the composed dg functor $\iota \circ \pi'$ and by $\epsilon: \pi \Rightarrow \pi \circ F$ the morphism of dg functors $\iota \circ \epsilon'$. Recall from  \cite[\S5.1]{Doc} that when the dg functor $F$ induces an equivalence of categories $\dgHo^0(F):\dgHo^0(\cA) \stackrel{\simeq}{\to} \dgHo^0(\cA)$, the morphism of dg functors $\epsilon$ is a (objectwise) quasi-isomorphism.
\section{Proof of the main result}\label{sec:proof}
The proof of Theorem \ref{thm:main} is divided in four steps:
\begin{itemize}
\item[(i)] Firstly, we construct a short exact sequence of dg categories relating $\cA/F^\bbN$ with the dg orbit category $\cA/F^\bbZ$. 
\item[(ii)] Secondly, we express the kernel of this above short exact sequence in terms of a square-zero extension $\cA \ltimes \mathrm{B}_1$.
\item[(iii)] Thirdly, making use of the theory of noncommutative motives, we relate the induced morphism $E(\cA \ltimes \mathrm{B}_1) \to E(\cA/F^\bbN)$ with $E(F) -\Id$.
\item[(iv)] Finally, using appropriate gradings of $\cA\ltimes \mathrm{B}_1$ and $\cA/F^\bbN$, we show that the above two morphisms agree.
\end{itemize}
\subsection*{Step I: Short exact sequence}
Recall from \S\ref{sec:orbit} that for every object $x \in \cA$ we have an associated degree-zero cycle $\epsilon'_x:x \to F(x)$ in $\cA/F^\bbN$. Let $\cA'$ be the full dg subcategory of $\cC_\dg(\cA/F^\bbN)$ consisting of the objects $\mathrm{cone}(\widehat{\epsilon'_x}), x \in \cA$; recall from \cite[Lem.~4.8]{BLL} that in any dg category the cone of a degree-zero cycle is unique up to unique dg isomorphism. We write $\mathrm{B}'$ for the associated~$\cA'\text{-}(\cA/F^\bbN)$-bimodule:
\begin{eqnarray*}
\cA' \otimes^\bfL (\cA/F^\bbN)^\op \too \cC_\dg(k) && (\mathrm{cone}(\widehat{\epsilon'_x}),y) \mapsto \cC_\dg(\cA/F^\bbN)(\widehat{y},\mathrm{cone}(\widehat{\epsilon'_x}))\,.
\end{eqnarray*}
Note that since $\widehat{x}$ and $\widehat{F(x)}$ belong to $\cD_c(\cA/F^\bbN)$, $\mathrm{B}'$ belongs to $\rep(\cA',\cA/F^\bbN)$.
\begin{proposition}
We have the following short exact sequence of dg categories
\begin{equation}\label{eq:ses}
0 \too \cA' \stackrel{\mathrm{B}'}{\too} \cA/F^\bbN \stackrel{{}_\iota\mathrm{B}}{\too} \cA/F^\bbZ \too 0
\end{equation}
in the homotopy category $\Hmo(k)$.
\end{proposition}
\begin{proof}
By definition of a short exact sequence of dg categories (see \cite[Thm.~4.11]{ICM-Keller}), we need to prove that the following sequence of triangulated categories
\begin{equation}\label{eq:sequence-1}
\cD_c(\cA') \hookrightarrow \cD_c(\cA/F^\bbN) \stackrel{\bfL \iota_\ast}{\too} \cD_c(\cA/F^\bbZ)
\end{equation}
is exact in the sense of Verdier \cite{Verdier}. Consider the following adjunction:
$$
\xymatrix{
\cD(\cA/F^\bbZ) \ar@<1ex>[d]^-{\iota^\ast} \\
\cD(\cA/F^\bbN) \ar@<1ex>[u]^-{\bfL \iota_\ast}\,. 
}
$$
We start by showing that $\iota^\ast$ is fully faithful. Since the functors $\iota^\ast$ and $\bfL \iota_\ast$ commute with infinite direct  sums, it suffices to show that the counit of the adjunction $\bfL \iota_\ast(\iota^\ast(\widehat{x})) \to \widehat{x}$ is an isomorphism for every $x \in \cA$. The object $\iota^\ast(\widehat{x}) \in \cD(\cA/F^\bbN)$ can be identified with the homotopy colimit of the following diagram:
\begin{equation}\label{eq:diagram-2}
\widehat{x} \too \cdots \too \widehat{F^p(x)} \stackrel{\widehat{\epsilon'_{F^p(x)}}}{\too} \widehat{F^{p+1}(x)} \too \cdots\,.
\end{equation}
Using the fact that $\bfL \iota_\ast(\widehat{F^p(x)})=\widehat{F^p(x)}$ and that $\bfL \iota_\ast(\widehat{\epsilon'_{F^p(x)}})=\widehat{\epsilon_{F^p(x)}}$ is an isomorphism, we hence conclude that the counit of the adjunction is an isomorphism. 

Let us write $\cN$ for the kernel of the functor $\bfL \iota_\ast$. Thanks to \cite[\S A.1 Lem.~(a)]{ilc}, we have the following exact sequence of triangulated categories:
\begin{equation}\label{eq:triang-cat-aux}
\cN \hookrightarrow \cD(\cA/F^\bbN) \stackrel{\bfL \iota_\ast}{\too} \cD(\cA/F^\bbZ)\,.
\end{equation}
We now show that $\cN=\cD(\cA')$. Since the functor $\bfL \iota_\ast$ commutes with infinite direct sums and $\bfL \iota_\ast(\widehat{\epsilon'_x})=\widehat{\epsilon_x}$ is an isomorphism, every object of $\cD(\cA')$ clearly belongs to $\cN$. In order to establish the converse inclusion $\cN \subseteq \cD(\cA')$, it suffices by Lemma \ref{lem:aux-1} below to show that the following objects 
\begin{eqnarray}\label{eq:objects1}
\mathrm{cone}(\widehat{x} \too \iota^\ast(\bfL \iota_\ast(\widehat{x}))) && x \in \cA
\end{eqnarray}
belong to $\cD(\cA')$. As mentioned above, $\iota^\ast(\bfL \iota_\ast(\widehat{x}))=\iota^\ast(\widehat{x})$ can be identified with the homotopy colimit of \eqref{eq:diagram-2}. Therefore, \eqref{eq:objects1} can be re-written as follows:
\begin{eqnarray}\label{eq:hocolim}
\mathrm{hocolim}_{p\geq 0} (\mathrm{cone}(\widehat{x} \too \widehat{F^p(x)})) && x \in \cA\,.
\end{eqnarray}
Using the octahedral axiom and the fact that the triangulated category $\cD(\cA')$ admits infinite direct sums, we hence conclude that the objects \eqref{eq:hocolim} belongs to $\cD(\cA')$. This implies that $\cN=\cD(\cA')$. Finally, by applying \cite[Thm.~2.1]{Neeman-1} to \eqref{eq:triang-cat-aux} (with $\cN$ replaced by $\cD(\cA')$), we obtain the searched sequence of triangulated categories \eqref{eq:sequence-1}. This achieves the proof.
\end{proof}
\begin{lemma}\label{lem:aux-1}
The triangulated category $\cN$ is generated by the following objects:
\begin{eqnarray*}
\mathrm{cone}(\widehat{x} \too \iota^\ast(\bfL \iota_\ast(\widehat{x}))) && x \in \cA\,.
\end{eqnarray*}
\end{lemma}
\begin{proof}
Since the triangulated categories $\cD(\cA/F^\bbN)$ and $\cD(\cA/F^\bbZ)$ admit infinite direct sums, $\iota^\ast$ commutes with infinite direct sums, and $\cD(\cA/F^\bbN)$ is generated by the objects $\widehat{x}, x \in \cA$, the proof is similar the one of \cite[\S A.1 Lem.~(b)]{ilc}.
\end{proof}
\subsection*{Step II: Square-zero extension}
Let $\cA$ be a dg category and $\mathrm{B}$ a $\cA\text{-}\cA$-bimodule.
\begin{definition}
The {\em square-zero extension $\cA \ltimes \mathrm{B}$ of $\cA$ by $\mathrm{B}$} is the dg category with the same objects as $\cA$ and complexes of $k$-modules $(\cA\ltimes \mathrm{B})(x,y) := \cA(x,y) \oplus \mathrm{B}(y,x)$. Given morphisms $(f,f') \in (\cA\ltimes \mathrm{B})(x,y)$ and $(g,g') \in (\cA\ltimes \mathrm{B})(y,z)$, the composition $(g,g')\circ (f,f')$ is defined as $(g \circ f, g'\cdot f + g \cdot f')$, where $\cdot$ stands for the $\cA\text{-}\cA$-bimodule structure of $\mathrm{B}$. Let us write $i: \cA \hookrightarrow \cA \ltimes \mathrm{B}$ for the inclusion dg functor.
\end{definition}
Let $F: \cA \to \cA$ be a dg functor which induces an equivalence of categories $\dgHo^0(F):\dgHo^0(\cA) \stackrel{\simeq}{\to} \dgHo^0(\cA)$. In this case we can consider the following $\cA\text{-}\cA$-bimodule:
\begin{eqnarray*}
\mathrm{B}_1: \cA\otimes^\bfL \cA^\op \too \cC_\dg(k) && (x,y) \mapsto \cA(y,F(x))[1]\,.
\end{eqnarray*} 
\begin{proposition}
We have the following Morita equivalence:
\begin{eqnarray}\label{eq:Morita}
\cA \ltimes \mathrm{B}_1 \too \cA' && x \mapsto \mathrm{cone}(\widehat{\epsilon'_x})\,.
\end{eqnarray}
\end{proposition}
\begin{proof}
Let $\cC^b(\cC(k))$ be the symmetric monoidal category of bounded cochain complexes in $\cC(k)$, and $\dgdgcat(k)$ the category of small $\cC^b(\cC(k))$-enriched categories. The symmetric monoidal totalization functor $\mathrm{Tot}: \cC^b(\cC(k)) \to \cC(k)$ gives rise to a well-defined functor $\mathrm{Tot}: \dgdgcat(k) \to \dgcat(k)$. Note that since we are using bounded cochain complexes in $\cC(k)$, there is no difference between the totalization functor $\mathrm{Tot}^\oplus$ and the totalization functor $\mathrm{Tot}^{\prod}$.

We start by introducing two auxiliar categories $\overline{\cA\ltimes \mathrm{B}_1}, \overline{\cA'} \in \dgdgcat(k)$. The first one has the same objects as $\cA$ and cochain complexes in $\cC(k)$ given by
$$(\overline{\cA\ltimes \mathrm{B}_1})(x,y):= \cdots \too0 \too \cA(x,y) \stackrel{0}{\too} \cA(x,F(y)) \too 0\too \cdots\,,$$
where $\cA(x,y)$ is of degree zero. The composition law is induced by the composition law of $\cA$ and by the following morphisms:
\begin{eqnarray*}
\cA(y,F(z)) \otimes \cA(x,y) \too \cA(x,F(z)) && (g',f) \mapsto g' \circ f \\
\cA(y,z) \otimes \cA(x, F(y)) \too \cA(x, F(z)) && (g, f') \mapsto F(g) \circ f\,.
\end{eqnarray*}
Note that $\mathrm{Tot}(\overline{\cA\ltimes \mathrm{B}_1})$ identifies with $\cA\ltimes \mathrm{B}_1$. In order to define the second auxiliar category $\overline{\cA'}$, we need to introduce some notations. Given objects $x, y \in \cA$, consider the following cochain complexes in the dg category $\cA/F^\bbN$
\begin{eqnarray*}
\cdots \too 0 \too x \stackrel{\epsilon'_x}{\too} F(x) \too 0\too \cdots && \cdots\too 0 \too y \stackrel{\epsilon'_y}{\too} F(y) \too 0 \too \cdots,
\end{eqnarray*}
where $x$ and $y$ are of degree zero. The associated cochain complex in $\cC(k)$ of morphisms from the left-hand side to the right-hand side is given by 
\begin{equation}\label{eq:complexes}
\cA/F^\bbN(F(x),y) \stackrel{d_{-1}}{\too} \cA/F^\bbN(x,y) \oplus \cA/F^\bbN(F(x),F(y)) \stackrel{d_0}{\too} \cA/F^\bbN(x,F(y))\,,
\end{equation}
where $d_{-1}(h):=(h \circ \epsilon'_x, \epsilon'_y \circ h)$ and $d_0(f,g):=\epsilon'_y \circ f - g \circ \epsilon'_x$. Under these notations, the auxiliar category $\overline{\cA'} \in \dgdgcat(k)$ is defined as having the same objects as $\cA$ and bounced cochain complexes in $\cC(k)$ given by $\overline{\cA'}(x,y):=\eqref{eq:complexes}$. The composition law is induced by the composition law of $\cA$. Via the Yoneda dg functor $\cA/F^\bbN \to \cC_\dg(\cA/F^\bbN)$, the totalization of \eqref{eq:complexes} can be identified with the cochain complex of $k$-modules in $\cC_\dg(\cA/F^\bbN)$ from $\mathrm{cone}(\widehat{\epsilon'_x})$ to $\mathrm{cone}(\widehat{\epsilon'_y})$; see \cite[\S2.4]{Drinfeld}. Consequently, we obtain the following identification of dg categories:
\begin{eqnarray}\label{eq:identification}
\mathrm{Tot}(\overline{\cA'}) \stackrel{\simeq}{\too} \cA' && x \mapsto \mathrm{cone}(\widehat{\epsilon'_x})\,.
\end{eqnarray}
Let us now relate the auxiliar categories $\overline{\cA\ltimes \mathrm{B}_1}$ and $\overline{\cA'}$. Given objects $x, y \in \cA$, consider the following morphism between bounded cochain complexes in $\cC(k)$:
\begin{equation}\label{eq:morphisms}
\xymatrix@C=1.1em@R=2.5em{
0 \ar[d] \ar[r] & \cA(x,y) \ar[r]^-0  \ar[d]^-{(\pi', \pi' \circ F)} & \cA(x,F(y)) \ar[d]^-{\pi'}\\
\cA/F^\bbN(F(x),y) \ar[r]^-{d_{-1}} & \cA/F^\bbN(x,y) \oplus \cA/F^\bbN(F(x),F(y)) \ar[r]^-{d_0} & \cA/F^\bbN(x,F(y));
}
\end{equation}
the commutativity of the right-hand side square follows from the that $\epsilon':\pi' \Rightarrow \pi' \circ F$ is a morphism of dg functors. Making use of the above morphism, we introduce the following $\cC^b(\cC(k))$-enriched functor:
\begin{eqnarray*}
\overline{\cA\ltimes \mathrm{B}_1} \too \overline{\cA'} & x \mapsto x & (\overline{\cA\ltimes \mathrm{B}_1})(x,y) \stackrel{\eqref{eq:morphisms}}{\too} \overline{\cA'}(x,y)\,.
\end{eqnarray*}
Thanks to Lemma \ref{lem:computation} below, the morphism \eqref{eq:morphisms} induce an isomorphism in horizontal cohomology. Consequently, its totalization is a quasi-isomorphism. This implies that the induced dg functor
\begin{eqnarray*}
\cA \ltimes \mathrm{B}_1 = \mathrm{Tot}(\overline{\cA \ltimes \mathrm{B}_1}) \too \mathrm{Tot}(\overline{\cA'})&& x \mapsto x
\end{eqnarray*}
is a Morita equivalence. By composing it with the above identification \eqref{eq:identification}, we hence obtain the searched Morita equivalence \eqref{eq:Morita}. This achieves the proof. 
\end{proof}
\begin{lemma}\label{lem:computation}
The above morphism \eqref{eq:morphisms}, between bounded cochain complexes in $\cC(k)$, induces an isomorphism in horizontal cohomology.
\end{lemma}
\begin{proof}
Given objects $x, y, z \in \cA/F^\bbN$, consider the following homomorpisms:
\begin{eqnarray}\label{eq:morphisms1}
(\cA/F^\bbN)(F(x),z) \too (\cA/F^\bbN)(x,z) && h \mapsto h \circ \epsilon'_x 
\end{eqnarray}
\begin{eqnarray}\label{eq:morphisms2}
(\cA/F^\bbN)(z,y) \too (\cA/F^\bbN)(z,F(y)) && h \mapsto \epsilon'_y \circ h \,.
\end{eqnarray}
Note that \eqref{eq:morphisms1}-\eqref{eq:morphisms2} correspond to the homomorphisms:
\begin{eqnarray*}
\bigoplus_{n\geq 0} \cA(F^{n+1}(x),z)\hookrightarrow\bigoplus_{n \geq 0}\cA(F^n(x),z) && \bigoplus_{n\geq 0} \cA(F^n(z),y) \stackrel{F}{\to} \bigoplus_{n \geq 0}\cA(F^n(z),F(y))\,.
\end{eqnarray*}
By taking $z=y$ in \eqref{eq:morphisms1}, we conclude that $d_{-1}$ is injective and consequently that the morphism \eqref{eq:morphisms} induces an isomorphism in $(-1)^{\mathrm{th}}$-cohomology. The above descriptions of \eqref{eq:morphisms1}-\eqref{eq:morphisms2} allow us also to conclude that the image of $d_0$ is given by $\bigoplus_{n \geq 1}\cA(F^n(x),F(y))$. Consequently, the morphism \eqref{eq:morphisms} also induces an isomorphism in $1^{\mathrm{th}}$-cohomology. In what concerns $0^{\mathrm{th}}$-cohomology, note that by taking $z=F(y)$, resp. $z=x$, in \eqref{eq:morphisms1}, resp. in \eqref{eq:morphisms2}, we conclude that 
$$ (\pi',\pi' \circ F): (\cA/F^\bbN)(x,y)\too (\cA/F^\bbN)(x,y)\oplus (\cA/F^\bbN)(F(x),F(y))$$
induces an isomorphism between $(\cA/F^\bbN)(x,y)$ and the kernel of $d_0$. Moreover, under such isomorphism, $d_{-1}$ identifies with the inclusion of $\bigoplus_{n \geq 0} \cA(F^{n+1}(x),y)$ into $\bigoplus_{n \geq 0} \cA(F^n(x),y)$. This implies that the morphism \eqref{eq:morphisms} also induces an isomorphism in $0^{\mathrm{th}}$-cohomology, and hence concludes the proof.
\end{proof}
\subsection*{Step III: Noncommutative motives}
Let $E:\dgcat(k) \to \cT$ be an $\bbA^1$-homotopy invariant. Thanks to the defining conditions (i)-(ii), the above short exact sequence of dg categories \eqref{eq:ses} gives rise to the following distinguished triangle:
\begin{equation}\label{eq:triangle}
E(\cA') \stackrel{E(\mathrm{B}')}{\too} E(\cA/F^\bbN) \stackrel{E(\iota)}{\too} E(\cA/F^\bbZ) \stackrel{\partial}{\too} \Sigma(E(\cA'))\,.
\end{equation}
\begin{proposition}\label{prop:diagram}
We have the following commutative diagram:
\begin{equation}\label{eq:com-square}
\xymatrix@C=3em@R=2em{
E(\cA') \ar[rr]^-{E(\mathrm{B}')} && E(\cA/F^\bbN) \\
E(\cA\ltimes \mathrm{B}_1) \ar[u]^-{E(\eqref{eq:Morita})}_-\sim && \\
E(\cA) \ar[u]^-{E(i)} \ar[rr]_-{E(F) -\Id} && E(\cA) \ar[uu]_-{E(\pi')}\,.
}
\end{equation}
\end{proposition}
\begin{proof}
Thanks to Lemma \ref{lem:universal} below, it suffices to prove Proposition \ref{prop:diagram} the particular case where $E=U$. Recall from \S\ref{sec:orbit} that we have a morphism of dg functors $\epsilon': \pi' \Rightarrow \pi' \circ F$. Let us denote by ${}_{\epsilon'}\mathrm{B}: {}_{\pi'}\mathrm{B} \Rightarrow {}_{\pi' \circ F}\mathrm{B}$ the induced morphism of $\cA\text{-}(\cA/F^\bbN)$-bimodules. By definition of the Morita equivalence \eqref{eq:Morita} and of the $\cA'\text{-}(\cA/F^\bbN)$-bimodule $\mathrm{B}'$, we observe that the composition $U(\mathrm{B}') \circ U(\eqref{eq:Morita}) \circ U(i)$ identifies with the Grothendieck class of the following $\cA\text{-}(\cA/F^\bbN)$-bimodule:
$$ \mathrm{cone}({}_{\epsilon'}\mathrm{B}: {}_{\pi'}\mathrm{B} \Rightarrow {}_{\pi' \circ F}\mathrm{B}) \in \rep(\cA,\cA/F^\bbN)\,.$$
Since the Grothendieck of this class is given by $[{}_{\pi' \circ F}\mathrm{B}]-[{}_{\pi'}\mathrm{B}]$, the proof follows then from the equalities $[{}_{\pi'}\mathrm{B}]=:U(\pi')$ and $[{}_{\pi' \circ F}\mathrm{B}] =:U(\pi' \circ F)$.
\end{proof}
\begin{lemma}\label{lem:universal}
Given an $\bbA^1$-homotopy invariant $E: \dgcat(k) \to \cT$, there is an (unique) additive functor $\overline{E}: \Hmo_0(k) \to \cT$ such that $\overline{E}\circ U\simeq E$. 
\end{lemma}
\begin{proof}
Recall from \cite{IMRN} that a functor $E:\dgcat(k) \to \mathrm{D}$, with values in an additive category, is called an {\em additive invariant} if it inverts the Morita equivalences and sends split short exact sequence of dg categories to direct sums. As proved in \cite[Thms.~5.3 and 6.3]{IMRN}, the functor $U:\dgcat(k) \to \Hmo_0(k)$ is the {\em universal additive invariant}, \ie given any additive category $\mathrm{D}$ there is an equivalence of categories
$$U^\ast: \mathrm{Fun}_{\mathrm{additive}}(\Hmo_0(k),\mathrm{D}) \stackrel{\simeq}{\too} \mathrm{Fun}_{\mathrm{add}}(\dgcat(k), \mathrm{D})\,,$$
where the left-hand side denotes the category of additive functors and the right-hand side the category of additive invariants. The proof follows now from the fact that every $\bbA^1$-homotopy invariant is also an additive invariant.
\end{proof}
\subsection*{Step IV: Gradings}
A dg category $\cB$ is called {\em $\bbN_0$-graded} if the complexes of $k$-modules $\cB(x,y)$ admit a direct sum decomposition $\bigoplus_{n \geq 0} \cB(x,y)_n$ which is preserved by the composition law. Let us denote by $\cB_0$ the associated dg category with the same objects as $\cB$ and $\cB_0(x,y):=\cB(x,y)_0$. Note that we have an inclusion dg functor $i_0:\cB_0 \hookrightarrow \cB$ and a projection dg functor $p: \cB \to \cB_0$ such that $p \circ i_0 =\Id$.
\begin{example}
\begin{itemize}
\item[(i)] The dg category $\cA/F^\bbN$, equipped with the direct sum decomposition \eqref{eq:decomp}, is $\bbN_0$-graded. In this case, $(\cA/F^\bbN)_0=\cA$ and $i_0=\pi'$.
\item[(ii)] The dg category $\cA\ltimes \mathrm{B}_1$, equipped with the direct sum decomposition given by $(\cA\ltimes \mathrm{B}_1)(x,y)_0:= \cA(x,y)$ and $(\cA\ltimes \mathrm{B}_1)(x,y)_1:= \mathrm{B}_1(y,x)$, is also $\bbN_0$-graded. In this case, $(\cA\ltimes \mathrm{B}_1)_0=\cA$ and $i_0=i$.
\end{itemize}
\end{example}
\begin{proposition}\label{prop:grading}
Every $\bbA^1$-homotopy invariant $E:\dgcat(k) \to \cT$ inverts the dg functors $i_0: \cB_0 \hookrightarrow \cB$. 
\end{proposition}
\begin{proof}
Since $p\circ i_0=\Id$, it remains to show that $E(i_0\circ p)=\Id$. Note that we have canonical dg functors 
$\mathrm{inc}: \cB \to \cB[t]$ and $\mathrm{ev}_0, \mathrm{ev}_1: \cB[t] \to \cB$ verifying the equalities $\mathrm{ev}_0 \circ \mathrm{inc} = \mathrm{ev}_1 \circ \mathrm{inc} = \Id$. Consider the following commutative diagram
\begin{equation}\label{eq:triangle-diagram}
\xymatrix{
&& \cB \\
\cB \ar@/_1.0pc/@{=}[drr] \ar[rr]^-H  \ar@/^1.0pc/[urr]^-{i_0 \circ p} && \cB[t] \ar[u]_-{\mathrm{ev}_0} \ar[d]^-{\mathrm{ev}_1} \\
&& \cB\,,
}
\end{equation}
where $H$ is the dg functor defined as
\begin{eqnarray*}
\cB \too \cB[t] &x \mapsto x & \cB(x,y)_n \hookrightarrow \cB(x,y)\otimes t^n\,.
\end{eqnarray*}
Since $E$ inverts the morphism $\mathrm{inc}$, it also inverts the morphisms $\mathrm{ev}_0$ and $\mathrm{ev}_1$. Therefore, by applying the functor $E$ to the above diagram \eqref{eq:triangle-diagram}, we conclude that $E(i_0 \circ p) = \Id$. This achieves the proof.
\end{proof}
\begin{remark}
Note that Proposition \ref{prop:grading} holds more generally for every functor $E:\dgcat(k) \to \cT$ which inverts the inclusion dg functors $\mathrm{inc}: \cB \to \cB[t]$.
\end{remark}
\subsection*{All together}
We now have all the necessary ingredients to finish the proof of Theorem \ref{thm:main}. Firstly, by applying Proposition \ref{prop:grading} to the dg functors $\pi': \cA \to \cA/F^\bbN$ and $i: \cA \to \cA \ltimes \mathrm{B}_1$, we obtain the isomorphisms
\begin{eqnarray*}
E(\pi'): E(\cA) \stackrel{\sim}{\too} E(\cA/F^\bbZ) && E(i): E(\cA) \stackrel{\sim}{\too} E(\cA\ltimes \mathrm{B}_1)\,.
\end{eqnarray*}
Secondly, by combining these isomorphisms with the above commutative diagram \eqref{eq:com-square} and triangle \eqref{eq:triangle}, we obtain the searched distinguished triangle \eqref{eq:triangle-searched}. This concludes the proof.
\section{Remaining proofs}\label{sec:remaining}
\subsection*{Proof of Corollary \ref{cor:main}}
Item (i) follows from the long exact sequence of homotopy groups associated to the distinguished triangle \eqref{eq:triangle-searched} (with $E=KH$) and from the fact that $K_n(\cA)=0$ for $n <0$. In what concerns item (ii), consider the following diagram:
\begin{equation}\label{eq:diagram}
\xymatrix@C=2.5em@R=1.5em{
K_0(\cA) \ar[d]_-\sim \ar[rr]^-{K_0(F)-\Id} && K_0(\cA) \ar[d]_-\sim \ar[rr]^-{K_0(\pi)} && K_0(\cA/F^\bbZ) \ar[d] \\
KH_0(\cA) \ar[rr]_-{KH_0(F)-\Id} && KH_0(\cA) \ar[rr]_-{KH_0(\pi)} && KH_0(\cA/F^\bbZ)\,.
}
\end{equation}
If $\dgHo^0(\cA/F^\bbZ)$ is triangulated and idempotent complete, then $K_0(\cA/F^\bbZ)$ agrees with the Grothendieck group of $\dgHo^0(\cA/F^\bbZ)$. Consequently, since the functor $\dgHo^0(\pi)$ is (essentially) surjective, the morphism $K_0(\pi)$ is surjective. Making use of item (i), we then conclude that the right-hand side morphism in \eqref{eq:diagram} is an isomorphism. 
\subsection*{Proof of Proposition \ref{prop:suspension}}
Recall from \S\ref{sec:dg} that we denote by ${}_{\Sigma^n}\mathrm{B}$, resp. by ${}_{\Id} \mathrm{B}$, the $\cA\text{-}\cA$-bimodule associated to the suspension dg functor $\Sigma^n:\cA \to \cA$, resp. to the identity dg functor $\Id: \cA \to \cA$. These $\cA\text{-}\cA$-bimodules are related by the following equality ${}_{\Sigma^n}\mathrm{B} = ({}_{\Id} \mathrm{B})[n] \in \rep(\cA,\cA)$. Consequently, we conclude that 
$$U(\Sigma^n)=[({}_{\Id} \mathrm{B})[n]]=(-1)^n[{}_{\Id} \mathrm{B}]=(-1)^nU(\Id) = (-1)^n\Id\,.$$
Making use of Lemma \ref{lem:universal}, this then implies that $E(\Sigma^n)=(-1)^n \Id$. 
\subsection*{Proof of Proposition \ref{prop:finite}}
The dg functors $s_i: k \to \cD^b_\dg(A), 1 \leq i \leq m$, associated to the simple right $A$-modules $S_1, \ldots, S_m$ give rise to an isomorphism
\begin{equation}\label{eq:iso-plus}
\bigoplus_{i=1}^m U(s_i) : \bigoplus_{i=1}^m U(k) \stackrel{\sim}{\too} U(\cD^b_\dg(A))\,;
\end{equation}
see \cite[Rk.~3.19]{Azumaya}. Therefore, the proof of item (i) follows from Lemma \ref{lem:universal}. In what concerns item (ii), note that by applying the functor $\Hom_{\Hmo_0(k)}(U(k),-)$ to 
\begin{equation}\label{eq:U(F)}
U(F):U(\cD^b_\dg(A)) \stackrel{\sim}{\too} U(\cD^b_\dg(A))
\end{equation}
we obtain the associated group automorphism of $K_0(\cD^b(A))$. Via \eqref{eq:iso-plus}, this group latter automorphism identifies with $M(F): \bigoplus_{i=1}^m \bbZ \stackrel{\sim}{\to} \bigoplus_{i=1}^m \bbZ$. Therefore, since $\mathrm{End}_{\Hmo_0(k)}(U(k))\simeq\bbZ$, we conclude by the Yoneda lemma that \eqref{eq:U(F)} is given by the matrix $M(F)$. The proof follows now once again from Lemma \ref{lem:universal}.
\subsection*{Proof of Proposition \ref{prop:computation}}
Thanks to Corollary \ref{prop:cluster}, $K_0(\underline{\mathrm{MCM}}(R_s))$ identifies with the cokernel of the endomorphism $C_{A_s}-\Id$ of $\bigoplus_{i=1}^s \bbZ$. Recall from Auslander-Reiten-Smal$\o$ \cite[page~289]{ARS} that $C_{A_s}-\Id$ corresponds to the following matrix:
\begin{equation}\label{eq:matrix}
\begin{pmatrix}
-2 & 1 & 0  &\cdots &0 \\
-1 & -1 & \ddots &\ddots & \vdots\\
-1 & 0 & \ddots & \ddots & 0 \\
 \vdots& \vdots& \ddots & \ddots & 1 \\
 -1 &0 & \cdots& 0 & -1
\end{pmatrix}_{s \times s}\,.
\end{equation}
The proof follows now from the fact that the cokernel of \eqref{eq:matrix} is isomorphic to $\bbZ/(s+1)$. A canonical generator is given by the image of the vector $(0,\cdots, 0, -1)$.
\subsection*{Proof of Proposition \ref{prop:computation2}}
Thanks to Corollary \ref{prop:cluster}, $K_0(\underline{\mathrm{MCM}}(R))$ identifies with the cokernel of the endomorphism $-C_Q-\Id$ of $\bbZ\oplus \bbZ$. This endomorphism corresponds to the matrix $
\begin{pmatrix}
-9 & 3 \\
-3 & 0
\end{pmatrix}$. Therefore, the proof follows from the fact that the cokernel of this matrix is isomorphic to $\bbZ/3 \times \bbZ/3$. Canonical generators are given by the image of the vectors $(1,0)$ and $(-3,-1)$.
\subsection*{Proof of Proposition \ref{prop:line}}
Thanks to Proposition \ref{prop:suspension}, it suffices to show that via \eqref{eq:Hd} the morphism $HP^\pm(-\otimes \cL)$ identifies with 
\begin{equation}\label{eq:morph-2}
-\cdot \mathrm{ch}(\cL): \bigoplus_n H^n_{dR}(X) \too \bigoplus_n H^n_{dR}(X)\,.
\end{equation}
The dg functor $-\otimes \cL$ can be written as the following Fourier-Mukai dg functor:
\begin{eqnarray*}
\Phi_{\Delta_\ast(\cL)}:\perf_\dg(X) \too \perf_\dg(X) && \cF \mapsto q_\ast(p^\ast(\cF) \otimes \Delta_\ast(\cL))\,.
\end{eqnarray*}
Consequently, as proved in \cite[Thm.~6.7]{Lefschetz}, the morphism \eqref{eq:morph-2} identifies with
\begin{eqnarray*}
\bigoplus_n H^n_{dR}(X) \too \bigoplus_n H^n_{dR}(X) && \alpha \mapsto q_\ast\left(p^\ast(\alpha)  \cdot \mathrm{ch}(\Delta_\ast(\cL)) \cdot \sqrt{\mathrm{Td}_{X \times X}}\right)\,,
\end{eqnarray*}
Let us now show that this morphism identifies with \eqref{eq:morph-2}. Consider the equalities:
\begin{eqnarray}
\mathrm{ch}(\Delta_\ast(\cL)) \cdot \mathrm{Td}_{X \times X} & = & \Delta_\ast(\mathrm{ch}(\cL) \cdot \mathrm{Td}_X) \label{eq:equality1} \\
& = & \Delta_\ast(\mathrm{ch}(\cL) \cdot \Delta^\ast(\sqrt{\mathrm{Td}_{X \times X}})) \label{eq:equality2} \\
& = & \Delta_\ast(\mathrm{ch}(\cL)) \cdot \sqrt{\mathrm{Td}_{X \times X}}\,. \label{eq:equality3}
\end{eqnarray}
Some explanations are in order: \eqref{eq:equality1} follows from the Grothendieck-Riemann-Roch theorem applied to the morphism $\Delta$, \eqref{eq:equality2} follows from the fact that $\mathrm{Td}_X = \Delta^\ast(\sqrt{\mathrm{Td}_{X \times X}})$, and \eqref{eq:equality3} follows from the projection formula. If we divide by $\sqrt{\mathrm{Td}_{X \times X}}$, we then obtain the equality $ \mathrm{ch}(\Delta_\ast(\cL)) \cdot \sqrt{\mathrm{Td}_{X \times X}} = \Delta_\ast(\mathrm{ch}(\cL))$. Consequently, the above morphism identifies with
\begin{eqnarray}
\alpha & \mapsto & q_\ast(p^\ast(\alpha) \cdot \Delta_\ast(\mathrm{ch}(\cL))) \nonumber \\
& = & q_\ast(\Delta_\ast(\Delta^\ast(p^\ast(\alpha))\cdot \mathrm{ch}(\cL))) = q_\ast(\Delta_\ast(\alpha \cdot \mathrm{ch}(\cL)))= \alpha \cdot \mathrm{ch}(\cL) \label{eq:equality5} \,
\end{eqnarray}
where \eqref{eq:equality5} follows from the projection formula and from the equalities $p \circ \Delta = q \circ \Delta = \Id$. This achieves the proof. 
\subsection*{Proof of Proposition \ref{prop:reduction}}
Recall from \cite[\S3]{Lefschetz} that we have the bijection 
\begin{eqnarray*}
\mathrm{Iso}\,\perf(X\times X) \stackrel{\sim}{\too} \Hom_{\Hmo(k)}(\perf_\dg(X), \perf_\dg(X)) && \cE \mapsto \Phi_\cE\,,
\end{eqnarray*}
where $\Phi_\cE$ stands for the Fourier-Mukai dg functor $ \cF \mapsto q_\ast(p^\ast(\cF) \otimes \cE)$. Consequently, the morphism $U(\Phi)$ identifies with the Grothendieck class of 
$$ \mathrm{cone}\left(q^\ast(\cE^\vee)\otimes p^\ast(\cE) \to \Delta_\ast(\cO_X)\right) \in \perf(X \times X)\,.$$
Since this class is given by $[\Delta_\ast(\cO_X)] - [q^\ast(\cE^\vee)\otimes p^\ast(\cE)]$ and $\Phi_{\Delta_\ast(\cO_X)}=\Id$, we hence conclude that $U(\Phi)=\Id-U(\Phi')$. Making use of Lemma \ref{lem:universal}, this then implies that $E(\Phi)=\Id-E(\Phi')$.
\subsection*{Mukai pairing}
Since the $k$-scheme $X$ is proper, all the complexes of $k$-vector spaces $\perf_\dg(X)(\cF,\cF')$ belongs to $\cD_c(k)$. Consequently, the bimodule \eqref{eq:bimodule-Id} (with $\cA=\perf_\dg(X)$) corresponds to a morphism in the homotopy category $\Hmo(k)$:
\begin{equation}\label{eq:morphism}
\perf_\dg(X) \otimes \perf_\dg(X)^\op \too k\,.
\end{equation}
By first pre-composing \eqref{eq:morphism} with the Morita equivalence 
\begin{eqnarray*}
\perf_\dg(X) \otimes \perf_\dg(X) \too \perf_\dg(X) \otimes \perf_\dg(X)^\op && (\cF,\cG) \mapsto (\cF,\cG^\vee)\,,
\end{eqnarray*} 
and then by applying the symmetric monoidal functor $HP^\pm$ (see \cite[\S8]{Lefschetz}), we obtain (via the identification \eqref{eq:Hd}) the searched Mukai pairing:
\begin{equation*}
\langle -, -\rangle : \bigoplus_n H^n_{dR}(X)\otimes \bigoplus_n H^n_{dR}(X) \too k\,.
\end{equation*}
\subsection*{Proof of Proposition \ref{prop:spherical}}
By an abuse of notation, let us still denote by $\cE$, resp. by $\cE^\vee$, the dg functor $k \to \perf_\dg(X)$, resp. $k \to \perf_\dg(X)^\op$, associated to the perfect complex $\cE$, resp. to $\cE^\vee$. Under these notations, the bimodule ${}_{\Phi'} \mathrm{B}$ associated to the Fourier-Mukai dg functor $\Phi'$ corresponds to the following composition
$$ \perf_\dg(X) \stackrel{{}_{\Id}\mathrm{B} \otimes {}_{\cE^\vee}\mathrm{B}\otimes {}_{\cE}\mathrm{B}}{\too} \perf_\dg(X) \otimes \perf_\dg(X)^\op \otimes \perf_\dg(X) \stackrel{\eqref{eq:morphism}\otimes {}_{\Id} \mathrm{B}}{\too} \perf_\dg(X)\,.$$
Therefore, by applying the symmetric monoidal functor $HP^\pm$ to this composition, we conclude that via \eqref{eq:Hd}, $HP^\pm(\Phi')$ agrees with the following homomorphism:
\begin{eqnarray*}
\bigoplus_n H^n_{dR}(X) \too \bigoplus_n H^n_{dR}(X) && \alpha \mapsto \langle \alpha, HP^\pm(\cE)\rangle \cdot HP^\pm(\cE)\,.
\end{eqnarray*}
The proof follows now from the equality $HP^\pm(\cE)=\mathrm{ch}(\cE) \cdot \sqrt{\mathrm{Td}_X}$; see \cite[Thm.~6.7]{Lefschetz}.

\end{document}